\documentclass{llncs}
\usepackage{version}
\usepackage{nim}

\title{Division by Zero in Non-involutive Meadows}
\author{J.A. Bergstra \and C.A. Middelburg}
\institute{Informatics Institute, Faculty of Science, University of
           Amsterdam, \\
           Science Park~904, 1098~XH Amsterdam, the Netherlands \\
           \email{J.A.Bergstra@uva.nl,C.A.Middelburg@uva.nl}}

\begin{document}
\maketitle

\begin{abstract}
Meadows have been proposed as alternatives for fields with a purely 
equational axiomatization.
At the basis of meadows lies the decision to make the multiplicative 
inverse operation total by imposing that the multiplicative inverse of 
zero is zero.
Thus, the multiplicative inverse operation of a meadow is an involution.
In this paper, we study `non-involutive meadows', i.e.\ variants of
meadows in which the multiplicative inverse of zero is not zero,
and pay special attention to non-involutive meadows in which the 
multiplicative inverse of zero is one.
\begin{classcode}
12E12, 12L12, 68Q65
\end{classcode}
\end{abstract}

\section{Introduction}
\label{sect-intro}

The primary mathematical structure for measurement and computation is
unquestionably a field.
However, fields do not have a purely equational axiomatization, not all
fields are total algebras, and the class of all total algebras that 
satisfy the axioms of a field is not a variety.
This means that the theory of abstract data types cannot use the axioms 
of a field in applications to number systems based on rational, real or
complex numbers.

In~\cite{BT07a}, meadows are proposed as alternatives for fields with a
purely equational axiomatization.
A meadow is a commutative ring with a multiplicative identity element
and a total multiplicative inverse operation satisfying two equations
which imply that the multiplicative inverse of zero is zero.
Thus, all meadows are total algebras and the class of all meadows is a 
variety.
At the basis of meadows lies the decision to make the multiplicative 
inverse operation total by imposing that the multiplicative inverse of 
zero is zero.
All fields in which the multiplicative inverse of zero is zero, called
zero-totalized fields, are meadows, but not conversely.
In 2009, we found in~\cite{Ono83a} that meadows were already introduced 
almost 35 years earlier in~\cite{Kom75a}, where they go by the name of 
desirable pseudo-fields.
This discovery was first reported in~\cite{BM09g}.

We expect the total multiplicative inverse operation of zero-totalized 
fields, which is conceptually and technically simpler than the 
conventional partial multiplicative inverse operation, to be useful in 
\pagebreak[2]
among other things mathematics education.
Recently, in a discussion about research and development in mathematics 
education (M. van den Heuvel-Panhuizen, personal communication, 
March~25, 2014), we came across the alternative where the multiplicative 
inverse operation is made total by imposing that the multiplicative 
inverse of zero is one (see~\cite[pp.~158--160]{NT00a}).
At first sight, this seems a poor alternative.
However, it turns out to be difficult to substantiate this without 
working out the details of the variants of meadows in which the 
multiplicative inverse of zero is one.

By imposing that the multiplicative inverse of zero is zero, the 
multiplicative inverse operation is made an involution.
Therefore, we coined the name non-involutive meadow for a variant of a
meadow in which the multiplicative inverse of zero is not zero and
the name $n$-based non-involutive meadow ($n > 0$) for a \linebreak[2]
non-involutive meadow in which the multiplicative inverse of zero is 
$n$.
We con\-sider both zero-based meadow and involutive meadow to be 
alternative names for a meadow.
In this paper, we work out the details of one-based non-involutive 
meadows and non-involutive meadows.

We will among other things give equational axiomatizations of 
one-based non-involutive meadows and non-involutive meadows.
The axiomatization of non-involutive meadows allows of a uniform 
treatment of $n$-based non-involutive meadows for all $n > 0$.
It remains an open question whether there exists an equational 
axiomatization of the total algebras that are either involutive 
meadows or non-involutive meadows.

This paper is organized as follows.
First, we survey the axioms for meadows and related results
(Section~\ref{sect-Md}).
Next, we give the axioms for one-based non-involutive meadows and 
present results concerning the connections of one-based non-involutive 
meadows with meadows, one-totalized fields in general, and the 
one-totalized field of rational numbers (Section~\ref{sect-NiMd1}).
Then, we give the axioms for non-involutive meadows and present 
generalizations of the main results from the previous section to 
$n$-based non-involutive meadows (Section~\ref{sect-NiMd}).
Finally, we make some concluding remarks (Section~\ref{sect-concl}).

\section{Meadows}
\label{sect-Md}

In this section, we give a finite equational specification of the class
of all meadows and present related results.
For proofs, the reader is referred to earlier papers in which meadows 
have been investigated.
Meadows has been proposed as alternatives for fields with a purely 
equational axiomatization in~\cite{BT07a}.
They have been further investigated in 
e.g.~\cite{BBP13a,BHT09a,BM09g,BR08a} and applied in 
e.g.~\cite{BB09b,BM09d,BPZ07a}.

A meadow is a commutative ring with a multiplicative identity element 
and a total multiplicative inverse operation satisfying two equations 
which imply that the multiplicative inverse of zero is zero.
Hence, the signature of meadows includes the signature of a commutative 
ring with a multiplicative identity element.

The signature of commutative rings with a multiplicative identity
element consists of the following constants and operators:
\pagebreak[2]
\begin{itemize}
\item
the \emph{additive identity} constant $0$;
\item
the \emph{multiplicative identity} constant $1$;
\item
the binary \emph{addition} operator ${} + {}$;
\item
the binary \emph{multiplication} operator ${} \mmul {}$;
\item
the unary \emph{additive inverse} operator $- {}$;
\end{itemize}
The signature of meadows consists of the constants and operators from 
the signature of commutative rings with a multiplicative identity 
element and in addition:
\begin{itemize}
\item
the unary \emph{zero-totalized multiplicative inverse} operator 
${}\minv$.
\end{itemize}
We write:
\begin{ldispl}
\begin{array}{@{}l@{\;}c@{\;}l@{}}
\sigcr & \mathrm{for} & \set{0,1,{} + {},{} \mmul {}, - {}}\;,
\\
\sigmd & \mathrm{for} & \sigcr \union \set{{}\minv}\;.
\end{array}
\end{ldispl}%

We assume that there are infinitely many variables, including $x$, $y$
and $z$.
Terms are build as usual.
We use infix notation for the binary operators, prefix notation for the
unary operator $- {}$, and postfix notation for the unary 
operator~${}\minv$.
We use the usual precedence convention to reduce the need for
parentheses.
We introduce subtraction, division, and squaring as abbreviations: 
$p - q$ abbreviates $p + (-q)$, 
$p \mdiv q$ abbreviates $p \mmul q\minv$, and
$p^2$ abbreviates $p \mmul p$.
For each non-negative natural number $n$, we write $\ul{n}$ for the
numeral for $n$.
That is, the term $\ul{n}$ is defined by induction on $n$ as follows:
$\ul{0} = 0$ and $\ul{n+1} = \ul{n} + 1$.

The constants and operators from the signature of meadows are adopted 
from rational arithmetic, which gives an appropriate intuition about 
these constants and operators.

A commutative ring with a multiplicative identity element is a total 
algebra over the signature $\sigcr$ that satisfies the equations given 
in Table~\ref{eqns-CR}.%
\begin{table}[!t]
\caption
{Axioms of a commutative ring with a multiplicative identity element}
\label{eqns-CR}
\begin{eqntbl}
\begin{eqncol}
(x + y) + z = x + (y + z)                                             \\
x + y = y + x                                                         \\
x + 0 = x                                                             \\
x + (-x) = 0
\end{eqncol}
\qquad\quad
\begin{eqncol}
(x \mmul y) \mmul z = x \mmul (y \mmul z)                             \\
x \mmul y = y \mmul x                                                 \\
x \mmul 1 = x                                                         \\
x \mmul (y + z) = x \mmul y + x \mmul z
\end{eqncol}
\end{eqntbl}
\end{table}
A \emph{meadow} is a total algebra over the signature $\sigmd$ that 
satisfies the equations given in Tables~\ref{eqns-CR}
and~\ref{eqns-minv}.%
\footnote
{Throughout the paper, we use the term total algebra instead of algebra
 to emphasis that we mean an algebra without partial operations.}
\begin{table}[!t]
\caption{Additional axioms for a meadow}
\label{eqns-minv}
\begin{eqntbl}
\begin{eqncol}
{} \\[-3ex]
(x\minv)\minv = x                                      \hfill (2.1) \\
x \mmul (x \mmul x\minv) = x                           \qquad (2.2)
\end{eqncol}
\end{eqntbl}
\end{table}
We write:
\begin{ldispl}
\begin{array}{@{}l@{\;}c@{\;}l@{}}
\eqnscr  &
\multicolumn{2}{@{}l@{}}
 {\mathrm{for\; the\; set\; of\; all\; equations\; in\; Table\;
          \ref{eqns-CR}}\;,}
\\
\eqnsinv  &
\multicolumn{2}{@{}l@{}}
 {\mathrm{for\; the\; set\; of\; all\; equations\; in\; Table\;
          \ref{eqns-minv}}\;,}
\\
\eqnsmd & \mathrm{for} & \eqnscr \union \eqnsinv\;.
\end{array}
\end{ldispl}%
Equation~(2.1) is called \emph{Ref}, for reflection, and 
equation~(2.2) is called \emph{Ril}, for restricted inverse law.

Equations making the nature of the multiplicative inverse operation
in meadows more clear are derivable from the equations $\eqnsmd$.
\begin{proposition}
\label{prop-Md-derivable}
The equations 
\begin{ldispl}
0\minv = 0\;, \quad
1\minv = 1\;, \quad
(- x)\minv = - (x\minv)\;, \quad
(x \mmul y)\minv = x\minv \mmul y\minv 
\end{ldispl}%
are derivable from the equations $\eqnsmd$.
\end{proposition}
\begin{proof}
Theorem~2.2 from~\cite{BT07a} is concerned with the derivability of the 
first equation and Proposition~2.8 from~\cite{BHT09a} is concerned with 
the derivability of the last two equations.
The derivability of the second equation is trivial.
\qed
\end{proof}

The advantage of working with a total multiplicative inverse operation
lies in the fact that conditions like
$x \neq 0$ in $x \neq 0 \Implies x \mmul x\minv = 1$ are not needed to 
guarantee meaning.

A \emph{non-trivial meadow} is a meadow that satisfies the 
\emph{separation axiom} 
\begin{ldispl}
0 \neq 1\;;
\end{ldispl}%
and 
a \emph{cancellation meadow} is a meadow that satisfies the 
\emph{cancellation axiom} 
\begin{ldispl}
x \neq 0 \And x \mmul y = x \mmul z \Implies y = z 
\end{ldispl}%
or, equivalently, 
the \emph{general inverse law} 
\begin{ldispl}
x \neq 0 \Implies x \mmul x\minv = 1\;.
\end{ldispl}%

A \emph{totalized field} is a total algebra over the signature $\sigmd$ 
that satisfies the equations $\eqnscr$, the separation axiom, and the 
general inverse law.
A \emph{zero-totalized field} is a totalized field that satisfies in 
addition the equation $0\minv = 0$.
\begin{proposition}
\label{prop-origin-eqnsinv}
The equations $\eqnsinv$ are derivable from the axiomatization of 
zero-totalized fields given above.
\end{proposition}
\begin{proof}
This is Lemma~2.5 from~\cite{BT07a}.
\qed
\end{proof}
The following is a corollary of Proposition~\ref{prop-origin-eqnsinv}.
\begin{corollary}
\label{corollary-cMd-F0}
The class of all non-trivial cancellation meadows and the class of all
zero-totalized fields are the same.
\end{corollary}

Not all non-trivial meadows are zero-totalized fields, e.g.\ the initial 
meadow is not a zero-totalized field.
Nevertheless, we have the following theorem.
\begin{theorem}
\label{theorem-Md-F0}
The equational theory of meadows and the equational theory of 
zero-totalized fields are the same.
\end{theorem}
\begin{proof}
This is Theorem~3.10 from~\cite{BHT09a}.
\qed
\end{proof}
Theorem~\ref{theorem-Md-F0} can be read as follows: $\eqnsmd$ is a 
finite basis for the equational theory of cancellation meadows.

As a consequence of Theorem~\ref{theorem-Md-F0}, the separation axiom 
and the cancellation axiom may be used to show that an equation is 
derivable from the equations~$\eqnsmd$.

The cancellation meadow that we are most interested in is $\Ratz$, the 
zero-totalized field of rational numbers.
$\Ratz$ differs from the field of rational numbers only in that the
multiplicative inverse of zero is zero.%
\begin{theorem}
\label{theorem-Ratz}
$\Ratz$ is the initial algebra among the total algebras over the 
signature $\sigmd$ that satisfy the equations
\begin{ldispl}
\eqnsmd \union \set{(1 + x^2 + y^2) \mmul (1 + x^2 + y^2)\minv = 1}\;.
\end{ldispl}%
\end{theorem}
\begin{proof}
This is Theorem~9 from~\cite{BM09g}. 
\qed
\end{proof}

The following is an outstanding question with regard to meadows:
does there exist an equational specification of the class of all meadows 
with less than $10$ equations?

\section{One-Based Non-involutive Meadows}
\label{sect-NiMd1}

By imposing that the multiplicative inverse of zero is zero, the 
multiplicative inverse operation of a meadow is made an involution.
Therefore, we coined the name non-involutive meadow for a variant of a
meadow in which the multiplicative inverse of zero is not zero and
the name one-based non-involutive meadow for a non-involutive meadow 
in which the multiplicative inverse of zero is one.
In this section, we give a finite equational specification of the class
of all one-based non-involutive meadows.
Moreover, we present results concerning the connections of one-based 
non-involutive meadows with meadows, one-totalized fields in general, 
and the one-totalized field of rational numbers.

A one-based non-involutive meadow is a commutative ring with a 
multiplicative identity element and a total multiplicative inverse 
operation satisfying four equations which imply that the multiplicative 
inverse of zero is one.

The signature of one-based non-involutive meadows consists of the 
constants and operators from the signature of commutative rings with a 
multiplicative identity element and in addition:
\begin{itemize}
\item
the unary \emph{one-totalized multiplicative inverse} operator 
${}\niminvi$.%
\footnote
{We use different symbols for the zero-totalized and one-totalized 
 multiplicative inverse operations to allow of defining these operations
in terms of each other.}
\end{itemize}
We write:
\begin{ldispl}
\begin{array}{@{}l@{\;}c@{\;}l@{}}
\signimdi & \mathrm{for} & \sigcr \union \set{{}\niminvi}\;.
\end{array}
\end{ldispl}%
We use postfix notation for the unary operator~${}\niminvi$.

A \emph{one-based non-involutive meadow} is a total algebra over the 
signature $\signimdi$ that satisfies the equations given in 
Tables~\ref{eqns-CR} and~\ref{eqns-niminvi}.%
\begin{table}[!t]
\caption{Additional axioms for a one-based non-involutive meadow}
\label{eqns-niminvi}
\begin{eqntbl}
\begin{eqncol}
{} \\[-3ex]
(x\niminvi)\niminvi = x + (1 - x \mmul x\niminvi)        \hfill (3.1) \\
x \mmul (x \mmul x\niminvi) = x                          \hfill (3.2) \\
x\niminvi \mmul (x\niminvi)\niminvi = 1                  \hfill (3.3) \\
(x \mmul (x\niminvi \mmul x\niminvi))\niminvi \mmul (x \mmul x\niminvi)
 = x                                                     \qquad (3.4)
\end{eqncol}
\end{eqntbl}
\end{table}
We write:
\begin{ldispl}
\begin{array}{@{}l@{\;}c@{\;}l@{}}
\eqnsniinvi  &
\multicolumn{2}{@{}l@{}}
 {\mathrm{for\; the\; set\; of\; all\; equations\; in\; Table\;
          \ref{eqns-niminvi}}\;,}
\\
\eqnsnimdi & \mathrm{for} & \eqnscr \union \eqnsniinvi\;.
\end{array}
\end{ldispl}%
Apart from the different symbols used for the multiplicative inverse 
operation, equation~(3.1) is \emph{Ref} adapted to one-totalization of 
the multiplicative inverse operation and equation~(3.2) is simply 
\emph{Ril}.
The counterpart of equation~(3.4), viz.\ 
$(x \mmul (x\minv \mmul x\minv))\minv \mmul (x \mmul x\minv) = x$,
is derivable from~(2.1) and~(2.2) and consequently does hold in meadows 
as well.
However, the counterpart of equation~(3.3), viz.\ 
$x\minv \mmul (x\minv)\minv = 1$, does not hold in meadows.

\begin{proposition}
\label{prop-derivable-eqn-niminvi-0}
The equations $0\niminvi = 1$ and $1\niminvi = 1$ are derivable from the 
equations $\eqnsnimdi$.
\end{proposition}
\begin{proof}
We have $0\niminvi \mmul (0\niminvi)\niminvi = 1$ by~(3.3) and 
$(0\niminvi)\niminvi = 1$ by~(3.1).
From these equations, it follows immediately that $0\niminvi = 1$.
We have $1\niminvi = 1$ by~(3.2).
\qed
\end{proof}

The zero-totalized multiplicative inverse operator can be explicitly
defined in terms of the one-totalized multiplicative inverse operator
by the equation $x\minv = x \mmul (x\niminvi \mmul x\niminvi)$ and
the one-totalized multiplicative inverse operator can be explicitly
defined in terms of the zero-totalized multiplicative inverse operator
by the equation $x\niminvi = x\minv + (1 - x \mmul x\minv)$.

The following two lemmas will be used in proofs of subsequent theorems.
\begin{lemma}
\label{lemma-defeqns-1}
The following is derivable from 
$\eqnsmd \union \set{x\niminvi = x\minv + (1 - x \mmul x\minv)}$ 
as well as
$\eqnsnimdi \union \set{x\minv = x \mmul (x\niminvi \mmul x\niminvi)}$:
\begin{ldispl}
x \mmul x\niminvi = x \mmul x\minv\;.
\end{ldispl}%
\end{lemma}
\begin{proof}
We have 
$x \mmul x\niminvi = x \mmul x\minv + (x - x \mmul (x \mmul x\minv))$ 
by $x\niminvi = x\minv + (1 - x \mmul x\minv)$.
From this equation, it follows by (2.2) that 
$x \mmul x\niminvi = x \mmul x\minv$.
We have 
$x \mmul x\minv = (x \mmul (x \mmul x\niminvi)) \mmul x\niminvi$ 
by $x\minv = x \mmul (x\niminvi \mmul x\niminvi)$.
From this equation, it follows by (3.2) that 
$x \mmul x\minv = x \mmul x\niminvi$.
\qed
\end{proof}
\begin{lemma}
\label{lemma-defeqns-2}
The conditional equations given in Table~\ref{forms-derivable} are 
derivable from
$\eqnsmd \union \set{x\niminvi = x\minv + (1 - x \mmul x\minv)}$:
\begin{table}[!t]
\caption{Formulas concerning ${}\niminvi$ and ${}\minv$}
\label{forms-derivable}
\begin{eqntbl}
\begin{eqncol}
{} \\[-3ex]
x \neq 0 \Implies x\niminvi = x\minv
\hfill~(4.1) \\  
x \neq 0 \Implies (x\niminvi)\niminvi = (x\minv)\minv
\hfill~(4.2) \\ 
x \neq 0 \Implies x \mmul (x\niminvi \mmul x\niminvi) = x\minv
\hfill~(4.3) \\ 
x \neq 0 \Implies (x \mmul (x\niminvi \mmul x\niminvi))\niminvi = x
\qquad~(4.4) \\ 
x = 0    \Implies x\niminvi = 1
\hfill~(4.5) \\ 
x = 0    \Implies (x\niminvi)\niminvi = 1
\hfill~(4.6) 
\end{eqncol}
\end{eqntbl}
\end{table}
\end{lemma}
\begin{proof}
Recall that the equations $0\minv = 0$ and $1\minv = 1$ are derivable 
from $\eqnsmd$.
By Theorem~\ref{theorem-Md-F0}, we may use the general inverse law 
(\emph{Gil}) and the separation axiom (\emph{Sep}) to prove derivability 
from $\eqnsmd$.
It follows from~(2.2) and \emph{Sep} that 
$x \neq 0 \Implies x\minv \neq 0\;$~(*).

The derivability of (4.1)--(4.6) from $\eqnsmd$ and the defining 
equation of ${}\niminvi$ is proved as follows:
\begin{itemize}
\sloppy
\item
(4.1) follows immediately from the defining equation of ${}\niminvi$ and
\emph{Gil};
\item
(4.2) follows immediately from~(*) and~(4.1);
\item
(4.3) follows immediately from~(4.1), (2.1), and~(2.2);
\item
(4.4) follows immediately from~(4.3), (*), (4.1), and~(2.1);
\item
(4.5) follows immediately from the defining equation of ${}\niminvi$ and 
$0\minv = 0$;
\item
(4.6) follows immediately from~(4.5), the defining equation of 
${}\niminvi$, and \mbox{$1\minv = 1$}. 
\qed
\end{itemize}
\end{proof}

Despite the different multiplicative inverse operators, $\eqnsmd$ and 
$\eqnsnimdi$ are essentially the same in a well-defined sense.
\begin{theorem}
\label{theorem-defeqv-Md-NiMd1}
$\eqnsmd$ is definitionally equivalent to $\eqnsnimdi$,%
\footnote
{The notion of definitional equivalence originates from~\cite{Bou65a},
 where it was introduced, in the setting of first-order theories, under 
 the name of synonymy.
 In~\cite{Tay79a}, the notion of definitional equivalence was introduced
 in the setting of equational theories under the ambiguous name of
 equivalence.
 An abridged version of~\cite{Tay79a} appears in~\cite{Gra08a}.}
i.e.
\begin{ldispl}
\eqnsmd \union \set{x\niminvi = x\minv + (1 - x \mmul x\minv)}
 \vdash
\eqnsnimdi \union \set{x\minv = x \mmul (x\niminvi \mmul x\niminvi)}
\\
\hfill \mathrm{and} \hfill \phantom{\,}
\\
\eqnsnimdi \union \set{x\minv = x \mmul (x\niminvi \mmul x\niminvi)}
 \vdash
\eqnsmd \union \set{x\niminvi = x\minv + (1 - x \mmul x\minv)}\;.
\end{ldispl}
\end{theorem}
\begin{proof}
By Theorem~\ref{theorem-Md-F0}, we may use the general inverse law 
(\emph{Gil}) to prove derivability from $\eqnsmd$.
Recall that the equation $0\minv = 0$ is derivable from $\eqnsmd$.
By Lemma~\ref{lemma-defeqns-1}, the equation 
$x \mmul x\niminvi = x \mmul x\minv\;$~(**) is derivable from $\eqnsmd$ 
and the defining equation of ${}\niminvi$.

The derivability of (3.1)--(3.4) and the defining equation of ${}\minv$ 
from (2.1)--(2.2) and the defining equation of ${}\niminvi$ is proved as 
follows:
\begin{itemize}
\item
if $x \neq 0$, then
(3.1) follows immediately from (4.2), (2.1), \emph{Gil}, and (**);
\\
if $x = 0$, then
(3.1) follows immediately from (4.6);
\item
(3.2) follows immediately from (**) and~(2.2);
\item
if $x \neq 0$, then
(3.3) follows immediately from (4.1), (4.2), (2.1), and~\emph{Gil};
\\
if $x = 0$, then
(3.3) follows immediately from (4.5) and~(4.6);
\item
if $x \neq 0$, then
(3.4) follows immediately from (4.4), (**), and~(2.2);
\\
if $x = 0$, then
(3.4) follows trivially;
\item
if $x \neq 0$, then
the defining equation of ${}\minv$ follows immediately from (4.3);
\\
if $x = 0$, then
the defining equation of ${}\minv$ follows immediately from 
$0\minv = 0$.
\end{itemize}
The derivability of (2.1)--(2.2) and the defining equation of 
${}\niminvi$ from (3.1)--(3.4) and the defining equation of ${}\minv$ is 
proved as follows:
\begin{itemize}
\item
(2.1) follows immediately from the defining equation of ${}\minv$ 
and~(3.4) (twice);
\item
(2.2) follows immediately from the defining equation of ${}\minv$ 
and~(3.2) (twice);
\item
the defining equation of ${}\niminvi$ follows immediately from the
the defining equation of ${}\minv$, (3.2), (3.3), and~(3.1).
\qed
\end{itemize}
\end{proof}

A \emph{non-trivial one-based non-involutive meadow} is a one-based 
non-involutive meadow that satisfies the separation axiom and 
a \emph{one-based non-involutive cancellation meadow} is a one-based 
non-involutive meadow that satisfies the cancellation axiom or, 
equivalently, $x \neq 0 \Implies x \mmul x\niminvi = 1$.

The following two lemmas will be used in the proof of a subsequent theorem.
\begin{lemma}
\label{lemma-CMd-NiCMd-1}
Let $\alpha$ be the mapping from the class of all meadows to the class 
of all one-based non-involutive meadows that maps each meadow $\cA$ to 
the restriction to $\signimdi$ of the unique expansion of $\cA$ for 
which $\eqnsmd \union \set{x\niminvi = x\minv + (1 - x \mmul x\minv)}$ 
holds.
Then:
\begin{enumerate}
\item 
$\alpha$ is a bijection;
\item 
the restriction of $\alpha$ to the class of all cancellation meadows is 
a bijection.
\end{enumerate}
\end{lemma}
\begin{proof}
Let $\alpha'$ be the mapping from the class of all one-based 
non-involutive meadows to the class of all meadows that maps each 
one-based non-involutive meadow $\cA'$ to 
the restriction to $\sigmd$ of the unique expansion of $\cA'$ for which  
$\eqnsnimdi \union \set{x\minv = x \mmul (x\niminvi \mmul x\niminvi)}$
holds.
Then $\alpha \circ \alpha'$ and $\alpha' \circ \alpha$ are identity
mappings by Theorem~\ref{theorem-defeqv-Md-NiMd1}.
Hence, $\alpha$ is a bijection.

By Lemma~\ref{lemma-defeqns-1}, for each meadow $\cA$, 
$x \mmul x\niminvi = x \mmul x\minv$ holds in the unique expansion of 
$\cA$ for which 
$\eqnsmd \union \set{x\niminvi = x\minv + (1 - x \mmul x\minv)}$ holds.
This implies that, for each meadow $\cA$, $\alpha(\cA)$ satisfies the 
cancellation axiom if $\cA$ satisfies it.
In other words, $\alpha$ maps each cancellation meadow to a one-based 
non-involutive cancellation meadow.
Similar remarks apply to the inverse of $\alpha$.
Hence, the restriction of $\alpha$ to the class of all cancellation 
meadows is a bijection.
\qed
\end{proof}
\begin{lemma}
\label{lemma-CMd-NiCMd-2}
Let $\epsilon$ be the mapping from the set of all equations between 
terms over $\signimdi$ to the set of all equations between terms over 
$\sigmd$ that is induced by the defining equation of ${}\niminvi$ and
let $\alpha$ be as in Lemma~\ref{lemma-CMd-NiCMd-1}.
Then:
\begin{enumerate}
\item 
for each equation $\phi$ between terms over $\signimdi$,
$\eqnsnimdi \vdash \phi$ iff $\eqnsmd \vdash \epsilon(\phi)$;
\item
for each meadow $\cA$ and equation $\phi$ between terms over 
$\signimdi$, 
$\alpha(\cA) \models \phi$ iff $\cA \models \epsilon(\phi)$.
\end{enumerate}
\end{lemma}
\begin{proof}
Let $\epsilon'$ be the mapping from the set of all equations between 
terms over $\sigmd$ to the set of all equations between terms over 
$\signimdi$ that is induced by the defining equation of ${}\minv$.
Then, by Theorem~\ref{theorem-defeqv-Md-NiMd1}:
\begin{itemize}
\item
for each equation $\phi$ between terms over $\signimdi$,
$\eqnsmd \vdash \epsilon(\phi)$ if $\eqnsnimdi \vdash \phi$;
\item
for each equation $\phi'$ between terms over $\sigmd$,
$\eqnsnimdi \vdash \epsilon'(\phi')$ if $\eqnsmd \vdash \phi'$;
\item
for each equation $\phi$ between terms over $\signimdi$,
$\eqnsnimdi \vdash \epsilon'(\epsilon(\phi)) \Iff \phi$.
\end{itemize}
From this it follows immediately that, for each equation $\phi$ between 
terms over $\signimdi$,
$\eqnsnimdi \vdash \phi$ iff $\eqnsmd \vdash \epsilon(\phi)$.

Let $\alpha'$ be as in the proof of Lemma~\ref{lemma-CMd-NiCMd-1}.
Then for each one-based non-involutive meadow $\cA'$ and equation $\phi$ 
between terms over $\signimdi$, 
$\cA' \models \phi$ iff $\alpha'(\cA') \models \epsilon(\phi)$ by the 
construction of $\alpha'(\cA')$.
From this and the fact that $\alpha'(\alpha(\cA)) = \cA$ for each meadow 
$\cA$, it follows that, for each meadow $\cA$ and equation $\phi$ 
between terms over $\signimdi$, 
$\alpha(\cA) \models \phi$ iff $\cA \models \epsilon(\phi)$.
\qed
\end{proof}

Recall that a totalized field is a total algebra over the signature 
$\sigmd$ that satisfies the equations $\eqnscr$, the separation axiom, 
and the general inverse law.
A \emph{one-totalized field} is a totalized field that satisfies in 
addition the equation $0\minv = 1$.
\begin{proposition}
\label{prop-origin-eqnsniinvi}
After replacing all occurrences of the operator $\niminvi$ with $\minv$, 
the equations $\eqnsniinvi$ are derivable from the axiomatization of 
one-totalized fields given above.
\end{proposition}
\begin{proof}
It follows from the general inverse law and the separation axiom that 
$x \neq 0 \Implies x\minv \neq 0$.
It follows from this and the general inverse law that 
$x \neq 0 \Implies (x\minv)\minv = x\;$~($\dagger$).

The derivability of (3.1)--(3.4) from $\eqnscr$, the separation axiom, 
the general inverse law, and $0\minv = 1$ is proved as follows:
\begin{itemize}
\item
if $x \neq 0$, then (3.1) follows immediately from ($\dagger$) and
\emph{Gil};
\\
if $x = 0$, then (3.1) follows immediately from $0\minv = 1$ and
\emph{Gil};
\item
if $x \neq 0$, then (3.2) follows immediately from \emph{Gil};
\\
if $x = 0$, then (3.2) follow trivially;
\item
if $x \neq 0$, then (3.3) follows immediately from ($\dagger$) and 
\emph{Gil};
\\
if $x = 0$, then (3.3) follows immediately from $0\minv = 1$ and
\emph{Gil};
\item
if $x \neq 0$, then (3.4) follows immediately from \emph{Gil} and 
($\dagger$);
\\
if $x = 0$, then (3.4) follows trivially.
\qed
\end{itemize}
\end{proof}
The following is a corollary of 
Proposition~\ref{prop-origin-eqnsniinvi}.
\begin{corollary}
\label{corollary-cNiMd1-F1}
Up to naming of the multiplicative inverse operation, the class of 
all non-trivial one-based non-involutive cancellation meadows and the 
class of all one-totalized fields are the same.
\end{corollary}

Not all non-trivial one-based non-involutive meadows are one-totalized 
fields, e.g.\ the initial one-based non-involutive meadow is not a 
one-totalized field.
Nevertheless, we have the following theorem.
\begin{theorem}
\label{theorem-NiMd1-F1}
Up to naming of the multiplicative inverse operation, the equational 
theory of one-based non-involutive meadows and the equational theory of 
one-totalized fields are the same.
\end{theorem}
\begin{proof}
Let $\epsilon$ be as in Lemma~\ref{lemma-CMd-NiCMd-2}.
By Lemmas~\ref{lemma-CMd-NiCMd-1}.2 and~\ref{lemma-CMd-NiCMd-2}.2, we 
have that, for each equation $\phi$ between terms over $\signimdi$,
$\phi$ holds in all one-based non-involutive cancellation meadows only 
if $\epsilon(\phi)$ holds in all cancellation meadows.
From this, Theorem~\ref{theorem-Md-F0} and 
Corollary~\ref{corollary-cMd-F0} it follows that, for each equation 
$\phi$ between terms over $\signimdi$, $\phi$ holds in all one-based 
non-involutive cancellation meadows only if $\epsilon(\phi)$ is 
derivable from $\eqnsmd$.
From this and Lemma~\ref{lemma-CMd-NiCMd-2}.1 it follows that, for each 
equation $\phi$ between terms over $\signimdi$, $\phi$ holds in all 
one-based non-involutive cancellation meadows only if $\phi$ is 
derivable from $\eqnsnimdi$.
Hence, the equational theory of one-based non-involutive meadows and the 
equational theory of one-based non-involutive cancellation meadows are 
the same.
From this and Corollary~\ref{corollary-cNiMd1-F1} it follows that the 
equational theory of one-based non-involutive meadows and the 
equational theory of one-totalized fields are the same.
\qed
\end{proof}
Theorem~\ref{theorem-NiMd1-F1} can be read as follows: $\eqnsnimdi$ is a 
finite basis for the equational theory of one-based non-involutive 
cancellation meadows.

As a consequence of Theorem~\ref{theorem-NiMd1-F1}, the separation axiom 
and the cancellation axiom may be used to show that an equation is 
derivable from the equations~$\eqnsnimdi$.

\begin{proposition}
\label{prop-NiMd1-derivable}
The equations
\begin{ldispl}
(- x)\niminvi =
 - (x\niminvi) \mmul (x \mmul x\niminvi) + (1 - x \mmul x\niminvi)\;, 
\\ 
(x \mmul y)\niminvi =
 (x\niminvi \mmul y\niminvi) \mmul 
 ((x \mmul x\niminvi) \mmul (y \mmul y\niminvi)) +
 (1 - (x \mmul x\niminvi) \mmul (y \mmul y\niminvi)) 
\end{ldispl}%
are derivable from the equations $\eqnsnimdi$.
\end{proposition}
\begin{proof}
Recall that the equation $0\niminvi = 1$ is derivable from $\eqnsnimdi$.
The conditional equation $x \neq 0 \Implies x \mmul x\niminvi = 1$, 
which will be called $\mathit{Gil}'$ below, is a variant of the general 
inverse law derivable from~(3.2) and the cancellation axiom.
\begin{itemize}
\item
if $x \neq 0$, then $-x \mmul (-x \mmul (-x)\niminvi) = -x$ by~(3.2) and 
$-x \mmul (-x \mmul -(x\niminvi)) = -x$ by~(3.2), hence
$(- x)\niminvi = - (x\niminvi)$ by the cancellation axiom, hence 
$(- x)\niminvi =
 - (x\niminvi) \mmul (x \mmul x\niminvi) + (1 - x \mmul x\niminvi)$ by
$\mathit{Gil}'$;
\\
if $x = 0$, then the equation reduces to $0\niminvi = 1$;
\item
if $x \neq 0$ and $y \neq 0$, then 
$(x \mmul y) \mmul ((x \mmul y) \mmul (x \mmul y)\niminvi) = x \mmul y$
by~(3.2) and 
$(x \mmul y) \mmul ((x \mmul y) \mmul (x\niminvi \mmul y\niminvi)) =
 x \mmul y$ by~(3.2), 
hence $(x \mmul y)\niminvi = x\niminvi \mmul y\niminvi$ by the 
cancellation axiom, hence 
$(x \mmul y)\niminvi =
  (x\niminvi \mmul y\niminvi) \mmul 
  ((x \mmul x\niminvi) \mmul (y \mmul y\niminvi)) +
  (1 - (x \mmul x\niminvi) \mmul (y \mmul y\niminvi))$
by $\mathit{Gil}'$;
\\
if $x = 0$ or $y = 0$, then the equation reduces to $0\niminvi = 1$.
\qed
\end{itemize}
\end{proof}

The one-based non-involutive cancellation meadow that we are most 
interested in is $\Rato$, the one-totalized field of rational numbers.
$\Rato$ differs from the field of rational numbers only in that the
multiplicative inverse of zero is one.%
\begin{theorem}
\label{theorem-Rato}
$\Rato$ is the initial algebra among the total algebras over the 
signature $\signimdi$ that satisfy the equations
\begin{ldispl}
\eqnsnimdi \union 
\set{(1 + x^2 + y^2) \mmul (1 + x^2 + y^2)\niminvi = 1}\;.
\end{ldispl}%
\end{theorem}
\begin{proof}
The proof goes as for Theorem~9 from~\cite{BM09g}.
\qed
\end{proof}

\section{Non-involutive Meadows}
\label{sect-NiMd}

Recall that we coined the name non-involutive meadow for a variant of a
meadow in which the multiplicative inverse of zero is not zero.
Thus, in a non-involutive meadow, the multiplicative inverse of zero can
be anything.
In this section, we give a finite equational specification of the class
of all non-involutive mead\-ows.
Moreover, we present generalizations of the main results from 
Section~\ref{sect-NiMd1} to \linebreak[2] $n$-based non-involutive 
meadows.
Because these generalizations turn out to present no additional 
complications, for most proofs, the reader is only informed about the 
main differences with the corresponding proofs from 
Section~\ref{sect-NiMd1}.

The signature of non-involutive meadows consists of the constants and 
operators from the signature of commutative rings with a multiplicative 
identity element and in addition:
\begin{itemize}
\item
the unary \emph{totalized multiplicative inverse} operator ${}\niminv$.
\end{itemize}
We write:
\begin{ldispl}
\begin{array}{@{}l@{\;}c@{\;}l@{}}
\signimd & \mathrm{for} & \sigcr \union \set{{}\niminv}\;.
\end{array}
\end{ldispl}%
We use postfix notation for the unary operator~${}\niminv$.

A \emph{non-involutive meadow} is a total algebra over the signature 
$\signimd$ that satisfies the equations given in Tables~\ref{eqns-CR}
and~\ref{eqns-niminv}.%
\begin{table}[!t]
\caption{Additional axioms for a non-involutive meadow}
\label{eqns-niminv}
\begin{eqntbl}
\begin{eqncol}
{} \\[-3ex]
0\niminv \mmul (x\niminv)\niminv = 
0\niminv \mmul x + (1 - x \mmul x\niminv)                \qquad (5.1) \\
x \mmul (x \mmul x\niminv) = x                           \hfill (5.2) \\
x\niminv \mmul (x\niminv)\niminv = 1                     \hfill (5.3) \\
(x \mmul (x\niminv \mmul x\niminv))\niminv \mmul (x \mmul x\niminv) = x
                                                         \hfill (5.4)
\end{eqncol}
\end{eqntbl}
\end{table}
An \emph{$n$-based non-involutive meadow} is a non-involutive meadow 
that satisfies the equation $0\niminv = \ul{n}$.
We write:
\begin{ldispl}
\begin{array}{@{}l@{\;}c@{\;}l@{}}
\eqnsniinv    &
\multicolumn{2}{@{}l@{}}
 {\mathrm{for\; the\; set\; of\; all\; equations\; in\; Table\;
          \ref{eqns-niminv}}\;,}
\\
\eqnsnimd     & \mathrm{for} & \eqnscr \union \eqnsniinv\;,
\\
\eqnsnimdn{n} & \mathrm{for} & \eqnsnimd \union \set{0\niminv = \ul{n}}\;.
\end{array}
\end{ldispl}%
Apart from the different symbols used for the multiplicative inverse 
operation, equation~(5.1) is \emph{Ref} adapted to the arbitrary 
totalization of the multiplicative inverse operation and equation~(5.2) 
is simply \emph{Ril}.

Notice that $\eqnsnimdi$ and $\eqnsnimdn{1}$ are different sets of 
equations.
However, both $\eqnsnimdi$ and $\eqnsnimdn{1}$ equationally define the 
class of all one-based non-involutive meadows.
\begin{proposition}
$\eqnsnimdi$ and $\eqnsnimdn{1}$ are deductively equivalent, i.e.\
\label{prop-eqv-NiMd1-NiMd1}
\begin{ldispl}
\eqnsnimdi \vdash \eqnsnimdn{1} 
\quad \mathrm{and} \quad
\eqnsnimdn{1} \vdash \eqnsnimdi\;. 
\end{ldispl}
\end{proposition}
\begin{proof}
To prove that $\eqnsnimdi \vdash \eqnsnimdn{1}$, it is sufficient to
prove that $0\niminvi = 1$ and~(5.1) are derivable from $\eqnsnimdi$.
By Proposition~\ref{prop-derivable-eqn-niminvi-0}, we have that
$\eqnsnimdi \vdash 0\niminvi = 1$.
From~(3.1) and $0\niminvi = 1$, (5.1) follows immediately.
To prove that $\eqnsnimdn{n} \vdash \eqnsnimdi$, it is sufficient to
prove that~(3.1) is derivable from $\eqnsnimdn{n}$.
From~(5.1) and $0\niminv = 1$, (3.1) follows immediately.
\qed
\end{proof}

A \emph{non-trivial} ($n$-\emph{based}) \emph{non-involutive meadow} is 
an ($n$-based) non-involutive meadow that satisfies the separation axiom 
and an ($n$-\emph{based}) \emph{non-involutive cancellation meadow} is 
an ($n$-based) non-involutive meadow that satisfies the cancellation 
axiom or, equivalently, $x \neq 0 \Implies x \mmul x\niminv = 1$.

Recall that a totalized field is a total algebra over the signature 
$\sigmd$ that satisfies the equations $\eqnscr$, the separation axiom, 
and the general inverse law.
A \emph{non-zero-totalized field} is a totalized field that satisfies 
the inequation $0\minv \neq 0$ and a \emph{$n$-totalized field} is a 
totalized field that satisfies the equation $0\minv = \ul{n}$.
\begin{proposition}
\label{prop-origin-eqnsniinv}
After replacing all occurrences of the operator $\niminv$ by $\minv$, 
the equations $\eqnsniinv$ are derivable from the axiomatization of 
non-zero-totalized fields given above.
\end{proposition}
\begin{proof}
The proof goes as for Proposition~\ref{prop-origin-eqnsniinvi}, with
$0\minv = 1$ everywhere replaced by $0\minv \neq 0$.
\qed
\end{proof}
%

For each $n > 0$, $n$-based non-involutive meadows have a lot of 
properties in common with zero-based non-involutive meadows.
\begin{theorem}
\label{theorem-defeqv-Md-NiMdn}
For each $n > 0$, $\eqnsmd$ is definitionally equivalent to 
$\eqnsnimdn{n}$, i.e.
\begin{ldispl}
\eqnsmd \union 
\set{x\niminv = x\minv + \ul{n} \mmul (1 - x \mmul x\minv)}
 \vdash
\eqnsnimdn{n} \union \set{x\minv = x \mmul (x\niminv \mmul x\niminv)}
\\
\hfill \mathrm{and} \hfill \phantom{\,}
\\
\eqnsnimdn{n} \union \set{x\minv = x \mmul (x\niminv \mmul x\niminv)}
 \vdash
\eqnsmd \union 
\set{x\niminv = x\minv + \ul{n} \mmul (1 - x \mmul x\minv)}\;.
\end{ldispl}
\end{theorem}
\begin{proof}
The proof goes essentially as for Theorem~\ref{theorem-defeqv-Md-NiMd1}.
The proof of the derivability of (5.1)--(5.4) and the defining equation 
of ${}\minv$ from (2.1)--(2.2) and the defining equation of ${}\niminv$ 
goes slightly different for each of these equations in the case $x = 0$.
\qed
\end{proof}

Not all non-trivial $n$-based non-involutive meadows are $n$-totalized 
fields, e.g.\ the initial $n$-based non-involutive meadow is not a 
$n$-totalized field.
Nevertheless, we have the following theorem.
\begin{theorem}
\label{theorem-NiMdn-Fn}
For each $n > 0$, up to naming of the multiplicative inverse operation, 
the equational theory of $n$-based non-involutive meadows and the 
equational theory of $n$-totalized fields are the same.
\end{theorem}
\begin{proof}
The proof goes as for Theorem~\ref{theorem-NiMd1-F1}.
\qed
\end{proof}
Theorem~\ref{theorem-NiMdn-Fn} can be read as follows: $\eqnsnimdn{n}$ 
is a finite basis for the equational theory of $n$-based non-involutive 
cancellation meadows.

As a consequence of Theorem~\ref{theorem-NiMdn-Fn}, the separation axiom 
and the cancellation axiom may be used to show that an equation is 
derivable from the equations $\eqnsnimdn{n}$.

\begin{proposition}
\label{prop-NiMd-derivable}
For each $n > 0$, the equations
\begin{ldispl}
0\niminv = \ul{n}\;, 
\quad\;\;
1\niminv = 1\;, 
\quad\;\;
(- x)\niminv =
- (x\niminv) \mmul (x \mmul x\niminv) + 
\ul{n} \mmul (1 - x \mmul x\niminv)\;, 
\\ 
(x \mmul y)\niminv =
x\niminv \mmul y\niminv \mmul 
((x \mmul x\niminv) \mmul (y \mmul y\niminv)) +
\ul{n} \mmul 
(1 - (x \mmul x\niminv) \mmul (y \mmul y\niminv))
\end{ldispl}%
are derivable from the equations $\eqnsnimdn{n}$.
\end{proposition}
\begin{proof}
The equation $0\niminv = \ul{n}$ belongs to $\eqnsnimdn{n}$.
We have $1\niminv = 1$ by~(5.2).
The proof for the last two equations goes as for 
Proposition~\ref{prop-NiMd1-derivable}.
\qed
\end{proof}

The $n$-based non-involutive cancellation meadow that we are most 
interested in is $\Ratn{n}$, the $n$-totalized field of rational 
numbers.
$\Ratn{n}$ differs from the field of rational numbers only in that the
multiplicative inverse of zero is $n$.%
\begin{theorem}
\label{theorem-Ratn}
$\Ratn{n}$ is the initial algebra among the total algebras over the 
signature $\signimd$ that satisfy the equations
\begin{ldispl}
\eqnsnimdn{n} \union 
\set{(1 + x^2 + y^2) \mmul (1 + x^2 + y^2)\niminv = 1}\;.
\end{ldispl}%
\end{theorem}
\begin{proof}
The proof goes as for Theorem~9 from~\cite{BM09g}.
\qed
\end{proof}

The following is an outstanding question with regard to non-involutive 
meadows: are the equational theory of non-involutive meadows and the 
equational theory of non-zero-totalized fields the same up to naming of 
the multiplicative inverse operation?

\section{Concluding Remarks}
\label{sect-concl}

We have worked out the details of one-based non-involutive meadows and 
non-involutive meadows.
We have given finite equational specifications of the class of all
one-based non-involutive meadows and the class of all non-involutive 
\linebreak[2] meadows.
We have presented results concerning the connections of one-based 
non-involutive meadows with meadows, one-totalized fields in general, 
and the one-totalized field of rational numbers and also generalizations 
of these results to $n$-based non-involutive meadows.

One-based non-involutive meadows and non-involutive meadows require more 
axioms than (zero-based/involutive) meadows.
We believe that the axioms of meadows are more easily memorized than the 
axioms of one-based non-involutive meadows and the axioms of 
non-involutive meadows.
Despite the differences, the axiomatizations of (zero-based/involutive) 
meadows and $n$-based non-involutive meadows ($n > 0$) are essentially 
the same (i.e.\ they are definitionally equivalent).
Moreover, the connections of $n$-based non-involutive meadows ($n > 0$) 
with $n$-totalized fields in general and the $n$-totalized field of 
rational numbers are essentially the same as the connections of 
(zero-based/involutive) meadows with zero-totalized fields in general 
and the zero-totalized field of rational numbers.

The equational specification of the class of all non-involutive meadows 
allows of a uniform treatment of $n$-based non-involutive meadows for 
all $n > 0$.
It is an open question whether there exists a finite equational 
specification of the class of all involutive and non-involutive meadows.

\bibliographystyle{splncs03}
\bibliography{MD}

\end{document}